\documentclass[11pt]{article}
\pdfoutput=1
\usepackage{amsfonts}
\usepackage{mathrsfs}
\usepackage{amssymb,amsmath,mathrsfs, amsthm}
\usepackage{amssymb}
\usepackage{graphicx}
\usepackage{indentfirst}
\usepackage{cite}
\usepackage[numbers,sort&compress]{natbib}
\usepackage{cases}
\usepackage[marginal]{footmisc}
\allowdisplaybreaks[4]

\numberwithin{equation}{section}

\newtheorem{theorem}{Theorem}[section]

\newtheorem{proposition}[theorem]{Proposition}
\theoremstyle{definition}
\newtheorem{definition}[theorem]{Definition}
\newtheorem{remark}[theorem]{Remark}

\theoremstyle{remark}

\date{}

\begin{document}
\title{Landsberg curvature of twisted product Finsler metric}
\author{Lize Bian,\quad Yong He$^{*}$,\quad Jianghui Han  \\
\footnotesize School of Mathematical Sciences, XinJiang Normal
University, Urumqi 830017, PR China}
\maketitle
\par\textbf{Abstract:} Let $(M_1,F_1)$ and $(M_2,F_2)$ be two Finsler manifolds. The twisted product Finsler metric of $F_{1}$ and $F_{2}$ is a Finsler metric $F=\sqrt{F_1^2+f^2F_2^2}$ endowed on the product manifold $M_1\times M_2$, where $f$ is a positive smooth function on $M_1\times M_2$. In this paper, Landsberg curvature and mean Landsberg curvature of the twisted product Finsler metric are derived. Necessary and sufficient conditions for the twisted product Finsler metric to be Landsberg metric or weakly Landsberg metric are obtained. And it is proved that every twisted product Finsler metric has relatively isotropic Landsberg (resp. relatively isotropic mean Landsberg) curvature if and only if it is a Landsberg (resp. weakly Landsberg) metric.
\footnote{$^*$ Corresponding author\\
E-mail address: heyong@xjnu.edu.cn}

\section{Introduction}
In Finsler geometry, there are several important non-Riemannian quantities so that Finsler geometry is much more colorful than Riemannian geometry. Therefore, it is meaningful to study non-Riemannian quantities. Cartan torsion $C$ is a primary non-Riemannian quantity. Differentiating $C$ along geodesics gives rise to Landsberg curvature $L$ \cite{S}. The trace of $C$ and $L$ are mean Cartan torsion $I$ and mean Landsberg curvature $J$, respectively. Mean Landsberg curvature $J$ is the rate of change of mean Cartan torsion $I$ along geodesics \cite{S}.

Berwald curvature is the most important non-Riemannian quantity. Berwald curvature of a Finsler metric always vanishes if and only if it is a Berwald metric \cite{S}. Landsberg metric, i.e. Landsberg curvature always vanishes, is more general than Berwald metric. It is known that every Berwald metric is a Landsberg metric, conversely it may not be true. Because on a Berwald manifold, all tangent spaces with the induced norm are linearly isometric to each other, while on a Landsberg manifold, all tangent spaces with the induced norm are all isometric.

In fact, searching non-Berwaldian Landsberg metric is called the unicorn problem in Finsler geometry \cite{Bao}. There is no regular non-Berwaldian Landsberg metric was found till now \cite{ZWL}. However, if the metric is almost regular (allowed to be singular in some directions), some non-Berwaldian Landsberg metrics were found in the past years. More specifically, Zhou, Wang and Li \cite{ZWL} constructed almost regular Landsberg general $(\alpha, \beta)$-metrics, Shen \cite{Shen} showed some examples of almost regular Landsberg $(\alpha, \beta)$-metrics, certainly, all of them are non-Berwaldian Landsberg metric.

In recent years, warped product are widely used to construct Riemannian metrics and Finsler metrics with special curvature properties. Warped product was first introduced by Bishop and O'Neill \cite{BO} to construct Riemannian manifolds with negative sectional curvature. The notion of warped product was extended to Finsler spaces by Asanov \cite{GSA1,GSA2}. Hushmandi and Rezaii \cite{HR1} obtained equivalent conditions for a proper (nontrivial) warped product Finsler metric to be a Landsberg metric, that is, two Finsler metrics that constitute the warped product Finsler metric, one is Landsberg metric, the other reduces to Riemannian metric, and an equation holds. Peyghan, Tayebi and Najafi \cite{PTN1} obtained necessary and sufficient conditions for a proper warped product Finsler metric to be a weakly Landsberg metric, similarly, one Finsler metric is weakly Landsberg metric, the other reduces to Riemannian metric, and an equation holds. And they proved that every doubly warped product Finsler metric with relatively isotropic Landsberg (resp. relatively isotropic mean Landsberg) curvature is Landsberg (resp. weakly Landsberg) metric. He and Zhong \cite{HC} proved that the doubly warped product complex Finsler manifold is complex Landsberg manifold if and only if its component manifolds are complex Landsberg manifolds and the complex warped function is positive constant.

Twisted product is a natural generalization of warped product. The notion of twisted product of Riemannian manifolds was first mentioned by Chen \cite{CBY} and was extended to pseudo-Riemannian case by Ponge and Reckziegel \cite{PR}. Chen \cite{C} investigated the twisted product CR-submanifolds of K\"ahler manifolds. Kozma, Peter and Shimada \cite{KPS} extended the twisted product to Finsler geometry and studied the Cartan connection, geodesics and completeness of the twisted product Finsler manifold. Peyghan, Tayebi and Nourmohammadi \cite{PTN2} obtained necessary and sufficient conditions for the twisted product Finsler metric to be Berwald metric or weakly Berwald metric, respectively. Most recently, Xiao and He et al. \cite{Xiao1} obtained necessary and sufficient conditions for a doubly twisted product complex Finsler manifold to be a complex Landsberg manifold. More details for the twisted product Finsler metric in \cite{Deng1,Deng2,Xiao2,Xiao3}.

Thus it is a natural question to find necessary and sufficient conditions for the twisted product Finsler metric to be a Landsberg metric or weakly Landsberg metric, our purpose of doing these is to study the possibility of constructing some special Finsler metrics such as Landsberg metric and weakly Landsberg metric. And we wonder what properties of the twisted product Finsler metric with relatively isotropic Landsberg (resp. relatively isotropic mean Landsberg) curvature are.

The paper is organised as follows. In Section 2, we recall some basic concepts and notations of Finsler geometry and adjust the index system of the twisted product Finsler metric. In Section 3, we will derive mean Cartan torsion coefficients of the twisted product Finsler metric. In Section 4 and Section 5, we will derive Landsberg curvature coefficients and mean Landsberg curvature coefficients of the twisted product Finsler metric, and give necessary and sufficient conditions for the twisted product Finsler metric to be Landsberg metric or weakly Landsberg metric. Finally, we prove that every twisted product Finsler metric has relatively isotropic Landsberg (resp. relatively isotropic mean Landsberg) curvature if and only if it is a Landsberg (resp. weakly Landsberg) metric.

\section{Preliminary}
In this section, we briefly recall some basic concepts and notations which we need in this paper.

Let $M$ be an n-dimensional $C^{\infty}$ manifold. Denote $x=(x^1, \cdots, x^n)$ the local coordinates on $M$, and $(x, y)=(x^1, \cdots, x^n, y^1, \cdots, y^n)$ the induced local coordinates on the tangent bundle $TM$ of $M$.
\begin{definition}\cite{BCS}~~A Finsler metric of $M$ is a function $F: TM\rightarrow[0, \infty)$ with the following properties

(i) $F$ is smooth on the slit tangent bundle $TM-\{\textrm{zero}~\textrm{section}\}$,

(ii) $F(x,\lambda y)=\lambda F(x,y)$ for all $\lambda>0$,

(iii) The $n\times n$ Hessian matrix

\begin{equation*}
(g_{ij})=\Big(\Big[\frac{1}{2}F^{2}\Big]_{y^{i}y^{j}}\Big)
\end{equation*}
is positive definite on $TM-\{\textrm{zero}~\textrm{section}\}$.
\end{definition}

Denote $(g^{ik})$ the inverse matrix of $(g_{kj})$, such that $g^{ik}g_{kj}=\delta^{i}_{~j}$. And $y_{i}=g_{ij}y^{j}$.

Every Finsler metric $F$ induces a spray $\mathbb{G}:=y^i\frac{\partial}{\partial x^i}-2\mathbb{G}^i\frac{\partial}{\partial y^i}$. The spray coefficients $\mathbb{G}^i$ are defined by

\begin{equation*}
  \mathbb{G}^i:=\frac{1}{4}g^{ip}\{[F^2]_{x^qy^p}y^q-[F^2]_{x^p}\}.
\end{equation*}
The geodesics of $F$ characterized by local coordinates are given by

\begin{equation*}
\frac{d^{2}x^{i}}{dt^{2}}+2\mathbb{G}^{i}(x,\frac{dx}{dt})=0.
\end{equation*}

To distinguish between Finslerian and Riemannian, Cartan torsion $C:=C_{ijk}dx^{i}\otimes dx^{j}\otimes dx^{k}$ was introduced by Finsler \cite{F} and emphasized by Cartan \cite{Cartan}. The Cartan torsion coefficients $C_{ijk}$ are defined by

\begin{equation*}
C_{ijk}:=\frac{1}{2}\frac{\partial g_{ij}}{\partial y^{k}}.
\end{equation*}
It is well known that $C_{ijk}=0$ if and only if $F$ is Riemannian \cite{S}. Note that for $g^{ij}$, we have $C^{ij}_{~~k}:=-\frac{1}{2}\frac{\partial g^{ij}}{\partial y^{k}}$.

Mean Cartan torsion $I:=I_{i}dx^{i}$ is the trace of $C$. The mean Cartan torsion coefficients $I_{i}$ are defined by
\begin{equation}\label{Ii}
  I_{i}:=C_{ijk}g^{jk}.
\end{equation}
By Deicke's theorem, $F$ is Riemannian if and only if $I_{i}=0$ \cite{S}.

Throughout this paper, denote $``~|~"$ the horizontal covariant derivatives, and denote $``~;~"$ the vertical covariant derivatives such that $T^{j}_{~i;s}=\frac{\partial T^{j}_{~i}}{\partial y^{s}}$.

Landsberg curvature $L:=L_{ijk}dx^{i}\otimes dx^{j}\otimes dx^{k}$ is the rate of change of $C$ along geodesics, such that $L_{ijk}:=C_{ijk|m}y^{m}$. And $L_{ijk}$ can also be defined by
\begin{equation}\label{Lijk}
  L_{ijk}:=-\frac{1}{2}FF_{y^{l}}\frac{\partial^{3}\mathbb{G}^{l}}{\partial y^{i}\partial y^{j}\partial y^{k}}=-\frac{1}{2}FF_{y^{l}}B^{~l}_{i~jk}=-\frac{1}{2}g_{lm}y^{m}B^{~l}_{i~jk}=-\frac{1}{2}y_{l}B^{~l}_{i~jk},
\end{equation}
where $B^{~l}_{i~jk}$ are Berwald curvature coefficients. A Finsler metric is called Landsberg metric if and only if $L_{ijk}=0$ \cite{S}. It is easy to see that every Berwald metric is a Landsberg metric.

Mean Landsberg curvature $J:=J_{i}dx^{i}$ is the rate of change of $I$ along geodesics, such that $J_{i}:=I_{i|m}y^{m}$. Also $J_{i}$ is defined by
\begin{equation}\label{Ji}
  J_{i}:=L_{ijk}g^{jk}.
\end{equation}
A Finsler metric is called weakly Landsberg metric if and only if $J_{i}=0$ \cite{S}. Clearly, a Landsberg metric must be a weakly Landsberg metric, the converse does not hold in general. However in the case of dimension two, any weakly Landsberg metric must be a Landsberg metric \cite{LS}.

A Finsler metric $F$ is said to be of relatively isotropic Landsberg curvature and relatively isotropic mean Landsberg curvature if $L_{ijk}$ and $J_{i}$ satisfy the following equations, respectively

\begin{gather}
  L_{ijk}+c(x)FC_{ijk}=0,\label{isoL}\\
  J_{i}+c(x)FI_{i}=0,\label{isoJ}
\end{gather}
where $c=c(x)$ is a scalar function on $M$ \cite{CSD}. As well known, every Landsberg (resp. weakly Landsberg) metric has relatively isotropic Landsberg (resp. relatively isotropic mean Landsberg) curvature, such as $c(x)=0$, however the converse does not hold in general.

Let $(M_1,F_1)$ and $(M_2,F_2)$ be two Finsler manifolds with $dim M_1=m$ and $dim M_2=n$, then $M=M_1\times M_2$ is a Finsler manifold with $dim M=m+n$.

Throughout this paper we use the natural product coordinates system on the product manifold $M=M_{1}\times M_{2}$. Let $(p_{0},q_{0})$ be a point on $M$, then there are coordinate charts $(U,x_{1})$ and $(V,x_{2})$ on $M_{1}$ and $M_{2}$, respectively, such that $p_{0}\in U$ and $q_{0}\in V$. Thus we obtain a coordinate chart $(W,x)$ on $M$ such that $W=U\times V$ and $(p_{0},q_{0})\in W$, and for all $(p,q)\in W$, $x(p,q)=(x_{1}(p),x_{2}(q))$, where $x_{1}=(x^{1},\ldots,x^{m})$, $x_{2}=(x^{m+1},\ldots,x^{m+n})$.

Let $TM_{1}$, $TM_{2}$ and $TM$ are the tangent bundles of $M_{1}$, $M_{2}$ and $M$, respectively. Let $\pi_1:M_1\times M_2\rightarrow M_1$ and $\pi_2:M_1\times M_2\rightarrow M_2$ are natural projection maps, then $d\pi_1:T(M_1\times M_2)\rightarrow TM_1$ and $d\pi_2: T(M_1\times M_2)\rightarrow TM_2$ are the tangent maps induced by $\pi_1$ and $\pi_2$, respectively. Note that $d\pi_{1}(x,y)=(x_{1},y_{1})$ and $d\pi_{2}(x,y)=(x_{2},y_{2})$ for every $y=(y_{1},y_{2})\in T_{x}(M_{1}\times M_{2})$ with $y_{1}=(y^{1},\ldots,y^{m})\in T_{x_{1}}M_{1}$ and $y_{2}=(y^{m+1},\ldots,y^{m+n})\in T_{x_{2}}M_{2}$. Denote $\widetilde{M_{1}}=TM_{1}-\{\textrm{zero}~\textrm{section}\}$, $\widetilde{M_{2}}=TM_{2}-\{\textrm{zero}~\textrm{section}\}$ and $\widetilde{M}=\widetilde{M_{1}}\times \widetilde{M_{2}}\subset T(M_{1}\times M_{2})-\{\textrm{zero}~\textrm{section}\}$.

On the product manifold $M=M_{1}\times M_{2}$, we consider the following metric.
\begin{definition}\cite{KPS} Let $(M_1,F_1)$ and $(M_2,F_2)$ be two Finsler manifolds, and $f:M\rightarrow R^{+}$ be a smooth function. The twisted product of Finsler metrics $F_{1}$ and $F_{2}$ is a Finsler metric $F:\widetilde{M}\rightarrow R^{+}$ defined by
\begin{equation}\label{F}
  F(x,y)=\sqrt{F_{1}^{2}(\pi_1(x),d\pi_1(y))+f^{2}(\pi_1(x),\pi_2(x))F_{2}^{2}(\pi_2(x),d\pi_2(y))},
\end{equation}
where $x\in M$, $y\in \widetilde{M}$, $f$ is called the twisted function. In brief, $F$ is called the twisted product Finsler metric.
\end{definition}

Notice that it is natural to ask $F$ to be defined on $\widetilde{M}$ rather than on $T(M_{1}\times M_{2})-\{zero~section\}$, or on $\widetilde{M}_{1}\times TM_{2}$, or on $TM_{1}\times \widetilde{M}_{2}$, since $F_{i}$ is smooth on $TM_{i}$ if and only if $F_{i}$ comes from a Riemannian metric of $M_{i}$ for $i=1, 2$.

If function $f$ only depends on $M_1$, $F$ becomes warped product Finsler metric. If function $f$ only depends on $M_2$, $F$ becomes product Finsler metric. Particularly, $F$ is trivial if $f$ is constant. Unless otherwise stated, this paper studies the twisted product Finsler metric in general sense.

In order to express distinctly, we adjust the system of index which is different from that in \cite{KPS}. Throughout this paper, lowercase Greek indices such as $\alpha,\beta$ etc., will run from $1$ to $m+n$, lowercase Latin indices such as $i,j$ etc., will run from $1$ to $m$, lowercase Latin indices with a prime such as $i',j'$ etc., will run from $m+1$ to $m+n$. Quantities associated to $(M_1,F_1)$ and $(M_2,F_2)$ are denoted with upper indices $1$ and $2$, respectively, such as $\mathop{C_{ijk}}\limits^{1}$ and $\mathop{C_{i'j'k'}}\limits^{2}$.

The local coordinates $(x,y)$ on $\widetilde{M}$ are transformed by the rulers

\begin{align*}
  &\bar{x}^{i}=\bar{x}^{i}(x^{1},\cdots,x^{m}),~~~\bar{x}^{i'}=\bar{x}^{i'}(x^{m+1},\cdots,x^{m+n}), \\
  &\bar{y}^{i}=\frac{\partial\bar{x}^{i}}{\partial x^{j}}y^{j},~~~\bar{y}^{i'}=\frac{\partial\bar{x}^{i'}}{\partial x^{j'}}y^{j'}.
\end{align*}
For $\frac{\partial}{\partial y^{\alpha}}$, we have

\begin{equation*}
\frac{\partial}{\partial y^{i}}=\frac{\partial\bar{x}^{j}}{\partial x^{i}}\frac{\partial}{\partial\bar{y}^{j}},~~~\frac{\partial}{\partial y^{i'}}=\frac{\partial\bar{x}^{j'}}{\partial x^{i'}}\frac{\partial}{\partial\bar{y}^{j'}}.
\end{equation*}

Denote $g_{ij}=[\frac{1}{2}F_{1}^{2}]_{y^{i}y^{j}}$ and $h_{i'j'}=[\frac{1}{2}F_{2}^{2}]_{y^{i'}y^{j'}}$. Therefore, the fundamental tensor matrix $(G_{\alpha\beta})$ of the twisted product Finsler metric is given by

\begin{equation*}
(G_{\alpha\beta})=\Big(\Big[\frac{1}{2}F^{2}\Big]_{y^{\alpha}y^{\beta}}\Big)=\left(
  \begin{array}{cc}
    g_{ij} &\quad 0 \\
    0 &\quad f^{2}h_{i'j'} \\
  \end{array}
\right)
\end{equation*}
and the inverse $(G^{\beta\alpha})$ is given by
\begin{equation}\label{g}
(G^{\beta\alpha})=\left(
     \begin{array}{cc}
       g^{ji} & 0 \\
       0 & f^{-2}h^{j'i'} \\
     \end{array}
   \right)
\end{equation}

\section{Cartan torsion of the twisted product Finsler metric}
In this section, we shall derive the mean Cartan torsion coefficients of the twisted product Finsler metric.
\begin{proposition}\cite{PTN2}
Let $F$ be the twisted product of Finsler metrics $F_{1}$ and $F_{2}$. Then the Cartan torsion coefficients $C_{\alpha\beta\gamma}$ of $F$ are given by

\begin{align}
&C_{ijk}=\mathop{C_{ijk}}\limits^{1},\label{Cijk1}\\
&C_{i'j'k'}=f^{2}\mathop{C_{i'j'k'}}\limits^{2},\label{Cijk2}\\
&C_{i'jk}=C_{ij'k}=C_{ijk'}=C_{i'j'k}=C_{i'jk'}=C_{ij'k'}=0.\label{Cijk3}
\end{align}
\end{proposition}

\begin{proposition}
Let $F$ be the twisted product of Finsler metrics $F_{1}$ and $F_{2}$. Then the mean Cartan torsion coefficients $I_{\alpha}$ of $F$ are given by

\begin{align}
&I_{i}=\mathop{I_{i}}\limits^{1},\label{Ii1}\\
&I_{i'}=\mathop{I_{i'}}\limits^{2}.\label{Ii2}
\end{align}
\end{proposition}

\begin{proof}
According to \eqref{Ii}, we have
\begin{equation}\label{II}
 I_{\alpha}=C_{\alpha\beta\gamma}G^{\beta\gamma}.
\end{equation}
By putting $\alpha=i$ in \eqref{II}, we get
\begin{equation}\label{III}
I_{i}=C_{ijk}G^{jk}+C_{ij'k}G^{j'k}+C_{ijk'}G^{jk'}+C_{ij'k'}G^{j'k'},
\end{equation}
substituting \eqref{g}, \eqref{Cijk1} and \eqref{Cijk3} into \eqref{III}, we obtain

\begin{equation*}
I_{i}=\mathop{C_{ijk}}\limits^{1}g^{jk}=\mathop{I_{i}}\limits^{1}.
\end{equation*}
Similarly, we can obtain \eqref{Ii2}.
\end{proof}

\section{Landsberg curvature of the twisted product Finsler metric}
It is known that a Landsberg metric must be of relatively isotropic Landsberg curvature, however the converse does not hold in general, we want to know whether the converse does hold or not for the twisted product Finsler metric. For this purpose, we shall first give necessary and sufficient conditions for the twisted product Finsler metric to be a Landsberg metric.
\begin{proposition}\cite{PTN2}
Let $F$ be the twisted product of Finsler metrics $F_{1}$ and $F_{2}$. Then the Berwald curvature coefficients $B^{~\lambda}_{\alpha~\beta\gamma}$ of $F$ are given by
\begin{align}
&B^{~l}_{i~jk}=\mathop{B^{~l}_{i~jk}}\limits^{1}+f\mathop{C^{lp}_{~~i;j;k}}\limits^{1}\frac{\partial f}{\partial x^{p}}F_{2}^{2},\label{Bijkl1}\\
&B^{~l}_{i~jk'}=B^{~l}_{i~k' j}=B^{~l}_{k'~ij}=2f\mathop{C^{lp}_{~~i;j}}\limits^{1}\frac{\partial f}{\partial x^{p}}y_{k'},\label{Bijkl2}\\
&B^{~l}_{i~j'k'}=B^{~l}_{j'~ik'}=B^{~l}_{j'~k'i}=2f\mathop{C^{lp}_{~~i}}\limits^{1}\frac{\partial f}{\partial x^{p}}h_{j'k'},\label{Bijkl3}\\
&B^{~l}_{i'~j'k'}=-2f\mathop{C_{i'j'k'}}\limits^{2}g^{lp}\frac{\partial f}{\partial x^{p}},\label{Bijkl4}\\
&B^{~l'}_{i'~j'k'}=\mathop{B^{~l'}_{i'~j'k'}}\limits^{2}+f^{-1}(\mathop{C^{l'p'}_{~~i';j';k'}}\limits^{2}F_{2}^{2}
+2\mathop{C^{l'p'}_{~~j';k'}}\limits^{2}y_{i'}+2\mathop{C^{l'p'}_{~~i';k'}}\limits^{2}y_{j'}+2\mathop{C^{l'p'}_{~~i';j'}}\limits^{2}y_{k'}\nonumber\\
&\qquad\quad\quad+2\mathop{C^{l'p'}_{~~i'}}\limits^{2}h_{j'k'}+2\mathop{C^{l'p'}_{~~j'}}\limits^{2}h_{i'k'}+2\mathop{C^{l'p'}_{~~k'}}\limits^{2}h_{i'j'}-2h^{l'p'}\mathop{C_{i'j'k'}}\limits^{2})\frac{\partial f}{\partial x^{p'}},\label{Bijkl5}\\
&B^{~l'}_{i~jk}=B^{~l'}_{i~jk'}=B^{~l'}_{i~k'j}=B^{~l'}_{k'~ij}=B^{~l'}_{i~j'k'}=B^{~l'}_{j'~ik'}=B^{~l'}_{j'~k'i}=0.\label{Bijkl6}
\end{align}
\end{proposition}

\begin{proposition}
Let $F$ be the twisted product of Finsler metrics $F_{1}$ and $F_{2}$. Then the Landsberg curvature coefficients $L_{\alpha\beta\gamma}$ of $F$ are given by
\begin{align}
&L_{ijk}=\mathop{L_{ijk}}\limits^{1}+f\mathop{C^{p}_{~ij;k}}\limits^{1}\frac{\partial f}{\partial x^{p}}F_{2}^{2},\label{Lijk1}\\
&L_{ijk'}=L_{ik'j}=L_{k'ij}=f\mathop{C^{p}_{~ij}}\limits^{1}\frac{\partial f}{\partial x^{p}}y_{k'},\label{Lijk2}\\
&L_{i'j'k'}=\mathop{L_{i'j'k'}}\limits^{2}+f\mathop{C_{i'j'k'}}\limits^{2}y^{p}\frac{\partial f}{\partial x^{p}}+f^{-1}(\mathop{C^{p'}_{~i'j';k'}}\limits^{2}F_{2}^{2}+\mathop{C^{p'}_{~j'k'}}\limits^{2}y_{i'}\nonumber\\
&\qquad\quad\quad+\mathop{C^{p'}_{~i'k'}}\limits^{2}y_{j'}+\mathop{C^{p'}_{~i'j'}}\limits^{2}y_{k'}+\mathop{C_{i'j'k'}}\limits^{2}y^{p'})\frac{\partial f}{\partial x^{p'}},\label{Lijk3}\\
&L_{ij'k'}=L_{j'ik'}=L_{j'k'i}=0.\label{Lijk4}
\end{align}
\end{proposition}

\begin{proof}
According to \eqref{Lijk}, we have
\begin{equation}\label{LL}
 L_{\alpha\beta\gamma}=-\frac{1}{2}y_{\lambda}B^{~\lambda}_{\alpha~\beta\gamma}.
\end{equation}
By putting $\alpha=i, \beta=j, \gamma=k$ in \eqref{LL}, we get
\begin{equation}\label{LLL}
  L_{ijk}=-\frac{1}{2}(y_{l}B^{~l}_{i~jk}+y_{l'}B^{~l'}_{i~jk}),
\end{equation}
substituting \eqref{Bijkl1} and \eqref{Bijkl6} into \eqref{LLL}, and notice that $y_{l}\mathop{C^{lp}_{~~i;j;k}}\limits^{1}=-2\mathop{C^{p}_{~ij;k}}\limits^{1}$, we obtain
\begin{align*}
L_{ijk}&=-\frac{1}{2}y_{l}(\mathop{B^{~l}_{i~jk}}\limits^{1}+f\mathop{C^{lp}_{~~i;j;k}}\limits^{1}\frac{\partial f}{\partial x^{p}}F_{2}^{2})\\
\,&=\mathop{L_{ijk}}\limits^{1}+f\mathop{C^{p}_{~ij;k}}\limits^{1}\frac{\partial f}{\partial x^{p}}F_{2}^{2}.
\end{align*}
Similarly, we can obtain \eqref{Lijk2}-\eqref{Lijk4}.
\end{proof}

\begin{theorem}
Let $F$ be the twisted product of Finsler metrics $F_{1}$ and $F_{2}$. Then $F$ is a Landsberg metric if and only if $F_{1}$ is Landsberg metric, $F_{2}$ is Riemannian and the equation $\mathop{C^{p}_{~ij}}\limits^{1}\frac{\partial f}{\partial x^{p}}=0$ holds.
\end{theorem}

\begin{proof}
$F$ is a Landsberg metric if and only if
\begin{equation}\label{thL}
L_{\alpha\beta\gamma}=0.
\end{equation}
According to Proposition 4.2, \eqref{thL} is equivalent to
\begin{numcases}{}
\mathop{L_{ijk}}\limits^{1}+f\mathop{C^{p}_{~ij;k}}\limits^{1}\frac{\partial f}{\partial x^{p}}F_{2}^{2}=0,\label{thL1}\\
f\mathop{C^{p}_{~ij}}\limits^{1}\frac{\partial f}{\partial x^{p}}y_{k'}=0,\label{thL2}\\
\mathop{L_{i'j'k'}}\limits^{2}+f\mathop{C_{i'j'k'}}\limits^{2}y^{p}\frac{\partial f}{\partial x^{p}}+f^{-1}(\mathop{C^{p'}_{~i'j';k'}}\limits^{2}F_{2}^{2}+\mathop{C^{p'}_{~j'k'}}\limits^{2}y_{i'}+\mathop{C^{p'}_{~i'k'}}\limits^{2}y_{j'}\nonumber\\
+\mathop{C^{p'}_{~i'j'}}\limits^{2}y_{k'}+\mathop{C_{i'j'k'}}\limits^{2}y^{p'})\frac{\partial f}{\partial x^{p'}}=0.\label{thL3}
\end{numcases}

Firstly, we prove the necessity. Since $f\neq0$, \eqref{thL2} implies that
\begin{equation}\label{thLf}
\mathop{C^{p}_{~ij}}\limits^{1}\frac{\partial f}{\partial x^{p}}=0.
\end{equation}
Differentiating \eqref{thLf} with respect to $y^{k}$, we get
\begin{equation}\label{thL;}
\mathop{C^{p}_{~ij;k}}\limits^{1}\frac{\partial f}{\partial x^{p}}=0,
\end{equation}
substituting \eqref{thL;} into \eqref{thL1}, then \eqref{thL1} reduces to $\mathop{L_{ijk}}\limits^{1}=0$, which means that $F_{1}$ is a Landsberg metric. Differentiating \eqref{thL3} with respect to $y^{i}$, we obtain
\begin{equation}\label{thLC}
\mathop{C_{i'j'k'}}\limits^{2}\frac{\partial f}{\partial x^{i}}=0.
\end{equation}
Since $\frac{\partial f}{\partial x^{i}}\neq0$, \eqref{thLC} indicates $\mathop{C_{i'j'k'}}\limits^{2}=0$, which means that $F_{2}$ is Riemannian.

Secondly, we prove the sufficiency. Assume that $F_{1}$ is Landsberg metric, $F_{2}$ is Riemannian and $\mathop{C^{p}_{~ij}}\limits^{1}\frac{\partial f}{\partial x^{p}}=0$, it is easy to know that \eqref{thL1}-\eqref{thL3} hold, which means that $F$ is a Landsberg metric.
\end{proof}

\begin{remark}
In Theorem 4.3, if the twisted function $f$ only depends on $M_{1}$, $F$ becomes a warped product Finsler metric, then Theorem 4.3 becomes Theorem 6 in \cite{PTN1}, in other words, Theorem 4.3 is the generalization of Theorem 6 in \cite{PTN1}.
\end{remark}

\begin{theorem}
Every twisted product Finsler metric has relatively isotropic Landsberg curvature if and only if it is a Landsberg metric.
\end{theorem}

\begin{proof}
The sufficiency is obvious, now we prove the necessity. $F$ has relatively isotropic Landsberg curvature if and only if
\begin{equation}\label{thisoL}
L_{\alpha\beta\gamma}+c(x)FC_{\alpha\beta\gamma}=0.
\end{equation}
According to Proposition 3.1 and Proposition 4.2, \eqref{thisoL} is equivalent to
\begin{numcases}{}
\mathop{L_{ijk}}\limits^{1}+f\mathop{C^{p}_{~ij;k}}\limits^{1}\frac{\partial f}{\partial x^{p}}F_{2}^{2}+c(x)F\mathop{C_{ijk}}\limits^{1}=0,\label{thisoL1}\\
f\mathop{C^{p}_{~ij}}\limits^{1}\frac{\partial f}{\partial x^{p}}y_{k'}=0,\label{thisoL2}\\
\mathop{L_{i'j'k'}}\limits^{2}+f\mathop{C_{i'j'k'}}\limits^{2}y^{p}\frac{\partial f}{\partial x^{p}}+f^{-1}(\mathop{C^{p'}_{~i'j';k'}}\limits^{2}F_{2}^{2}+\mathop{C^{p'}_{~j'k'}}\limits^{2}y_{i'}+\mathop{C^{p'}_{~i'k'}}\limits^{2}y_{j'}\nonumber\\
+\mathop{C^{p'}_{~i'j'}}\limits^{2}y_{k'}+\mathop{C_{i'j'k'}}\limits^{2}y^{p'})\frac{\partial f}{\partial x^{p'}}+c(x)Ff^{2}\mathop{C_{i'j'k'}}\limits^{2}=0.\label{thisoL3}
\end{numcases}
 Differentiating \eqref{thisoL1} with respect to $y^{i'}$, we get
\begin{equation}\label{thisoL11}
2f\mathop{C^{p}_{~ij;k}}\limits^{1}\frac{\partial f}{\partial x^{p}}y_{i'}+c(x)f^{2}F^{-1}y_{i'}\mathop{C_{ijk}}\limits^{1}=0,
\end{equation}
contracting \eqref{thisoL11} with $y^{i'}$, we get
\begin{equation}\label{thisoL111}
2f\mathop{C^{p}_{~ij;k}}\limits^{1}\frac{\partial f}{\partial x^{p}}F_{2}^{2}+c(x)f^{2}F^{-1}F_{2}^{2}\mathop{C_{ijk}}\limits^{1}=0.
\end{equation}
Transforming \eqref{thisoL111}, then substituting it into \eqref{thisoL1}, \eqref{thisoL1} can be simplified as
\begin{equation}\label{thisoL1111}
\mathop{L_{ijk}}\limits^{1}+c(x)(F-\frac{1}{2}f^{2}F^{-1}F_{2}^{2})\mathop{C_{ijk}}\limits^{1}=0.
\end{equation}
Differentiating \eqref{thisoL1111} with respect to $y^{i'}$, we obtain
\begin{equation}\label{thisoL11111}
c(x)\frac{1}{2}f^{4}F^{-3}F_{2}^{2}y_{i'}\mathop{C_{ijk}}\limits^{1}=0,
\end{equation}
contracting \eqref{thisoL11111} with $y^{i'}$, we obtain
\begin{equation}\label{thisoL111111}
c(x)\frac{1}{2}f^{4}F^{-3}F_{2}^{4}\mathop{C_{ijk}}\limits^{1}=0.
\end{equation}
Similarly, by \eqref{thisoL3}, we can obtain
\begin{equation}\label{thisoL1111111}
c(x)f^{4}F^{-3}F_{1}^{2}F_{2}^{2}\mathop{C_{i'j'k'}}\limits^{2}=0.
\end{equation}
Notice that $F$ is non-Riemannian, according to \eqref{thisoL111111} and \eqref{thisoL1111111}, hence $c(x)=0$, which means that $F$ is a Landsberg metric.
\end{proof}

\section{Mean Landsberg curvature of the twisted product Finsler metric}
It is known that a weakly Landsberg metric must be of relatively isotropic mean Landsberg curvature, however the converse does not hold in general, similarly we want to know whether the converse does hold or not for the twisted product Finsler metric. For this purpose, we first give necessary and sufficient conditions for the twisted product Finsler metric to be a weakly Landsberg metric.
\begin{proposition}
Let $F$ be the twisted product of Finsler metrics $F_{1}$ and $F_{2}$. Then the mean Landsberg curvature coefficients $J_{\alpha}$ of $F$ are given by
\begin{align}
&J_{i}=\mathop{J_{i}}\limits^{1}+f\mathop{I^{p}_{~;i}}\limits^{1}\frac{\partial f}{\partial x^{p}}F_{2}^{2},\label{Ji1}\\
&J_{i'}=f^{-2}\mathop{J_{i'}}\limits^{2}+(f\mathop{I^{p}}\limits^{1}y_{i'}+f^{-1}\mathop{I_{i'}}\limits^{2}y^{p})\frac{\partial f}{\partial x^{p}}+f^{-3}(\mathop{I^{p'}_{~;i'}}\limits^{2}F_{2}^{2}+
\mathop{I^{p'}}\limits^{2}y_{i'}+\mathop{I_{i'}}\limits^{2}y^{p'})\frac{\partial f}{\partial x^{p'}}.\label{Ji2}
\end{align}
\end{proposition}

\begin{proof}
According to \eqref{Ji}, we have
\begin{equation}\label{JJ}
 J_{\alpha}=L_{\alpha\beta\gamma}G^{\beta\gamma}.
\end{equation}
By putting $\alpha=i$ in \eqref{JJ}, we get
\begin{equation}\label{JJJ}
J_{i}=L_{ijk}G^{jk}+L_{ij'k}G^{j'k}+L_{ijk'}G^{jk'}+L_{ij'k'}G^{j'k'},
\end{equation}
substituting \eqref{g}, \eqref{Lijk1} and \eqref{Lijk4} into \eqref{JJJ}, we obtain
\begin{equation*}
J_{i}=(\mathop{L_{ijk}}\limits^{1}+f\mathop{C^{p}_{~jk;i}}\limits^{1}\frac{\partial f}{\partial x^{p}}F_{2}^{2})g^{jk}=\mathop{J_{i}}\limits^{1}+f\mathop{I^{p}_{~;i}}\limits^{1}\frac{\partial f}{\partial x^{p}}F_{2}^{2}.
\end{equation*}
Similarly, we can obtain \eqref{Ji2}.
\end{proof}

\begin{theorem}
Let $F$ be the twisted product of Finsler metrics $F_{1}$ and $F_{2}$. Then $F$ is a weakly Landsberg metric if and only if $F_{1}$ is weakly Landsberg metric, $F_{2}$ is Riemannian and the equation $\mathop{I^{p}}\limits^{1}\frac{\partial f}{\partial x^{p}}=0$ holds.
\end{theorem}

\begin{proof}
$F$ is a weakly Landsberg metric if and only if
\begin{equation}\label{thJ}
J_{\alpha}=0.
\end{equation}
According to Proposition 5.1, \eqref{thJ} is equivalent to
\begin{numcases}{}
\mathop{J_{i}}\limits^{1}+f\mathop{I^{p}_{~;i}}\limits^{1}\frac{\partial f}{\partial x^{p}}F_{2}^{2}=0,\label{thJ3}\\
f^{-2}\mathop{J_{i'}}\limits^{2}+(f\mathop{I^{p}}\limits^{1}y_{i'}+f^{-1}\mathop{I_{i'}}\limits^{2}y^{p})\frac{\partial f}{\partial x^{p}}+f^{-3}(\mathop{I^{p'}_{~;i'}}\limits^{2}F_{2}^{2}+
\mathop{I^{p'}}\limits^{2}y_{i'}+\mathop{I_{i'}}\limits^{2}y^{p'})\frac{\partial f}{\partial x^{p'}}=0.\label{thJ4}
\end{numcases}

Firstly, we prove the necessity. Differentiating \eqref{thJ4} with respect to $y^{i}$, we get
\begin{equation}\label{thJI}
(f\mathop{I^{p}_{~;i}}\limits^{1}y_{i'}+f^{-1}\mathop{I_{i'}}\limits^{2}\delta^{p}_{~i})\frac{\partial f}{\partial x^{p}}=0,
\end{equation}
differentiating \eqref{thJI} with respect to $y^{j}$, we can get
\begin{equation}\label{thJII}
\mathop{I^{p}_{~;i;j}}\limits^{1}\frac{\partial f}{\partial x^{p}}=0,
\end{equation}
contracting \eqref{thJII} with $y^{j}$, we obtain
\begin{equation}\label{thJIII}
-2\mathop{I^{p}_{~;i}}\limits^{1}\frac{\partial f}{\partial x^{p}}=0.
\end{equation}
By using \eqref{thJIII}, \eqref{thJ3} reduces to $\mathop{J_{i}}\limits^{1}=0$, hence $F_{1}$ is a weakly Landsberg metric. Similarly, using \eqref{thJIII}, \eqref{thJI} can be simplified as
\begin{equation}\label{thJf}
\mathop{I_{i'}}\limits^{2}\frac{\partial f}{\partial x^{i}}=0.
\end{equation}
Since $\frac{\partial f}{\partial x^{i}}\neq0$, \eqref{thJf} indicates $\mathop{I_{i'}}\limits^{2}=0$, which means that $F_{2}$ is Riemannian. Contracting \eqref{thJIII} with $y^{i}$, we obtain
\begin{equation*}
2\mathop{I^{p}}\limits^{1}\frac{\partial f}{\partial x^{p}}=0.
\end{equation*}

Secondly, we prove the sufficiency. Assume that $F_{1}$ is weakly Landsberg metric, $F_{2}$ is Riemannian and $\mathop{I^{p}}\limits^{1}\frac{\partial f}{\partial x^{p}}=0$, it is easy to know that \eqref{thJ3}-\eqref{thJ4} hold, which means that $F$ is a weakly Landsberg metric.
\end{proof}

\begin{remark}
In Theorem 5.2, if the twisted function $f$ only depends on $M_{1}$, $F$ becomes a warped product Finsler metric, then Theorem 5.2 becomes Theorem 8 in \cite{PTN1}, in other words, Theorem 5.2 is the generalization of Theorem 8 in \cite{PTN1}.
\end{remark}

\begin{theorem}
Every twisted product Finsler metric has relatively isotropic mean Landsberg curvature if and only if it is a weakly Landsberg metric.
\end{theorem}

\begin{proof}
The sufficiency is obvious, now we prove the necessity. $F$ has relatively isotropic mean Landsberg curvature if and only if
\begin{equation}\label{thisoJ}
J_{\alpha}+c(x)FI_{\alpha}=0,
\end{equation}
According to Proposition 3.2 and Proposition 5.1, \eqref{thisoJ} is equivalent to
\begin{numcases}{}
\mathop{J_{i}}\limits^{1}+f\mathop{I^{p}_{~;i}}\limits^{1}\frac{\partial f}{\partial x^{p}}F_{2}^{2}+c(x)F\mathop{I_{i}}\limits^{1}=0,\label{thisoJ1}\\
\mathop{J_{i'}}\limits^{2}+f(f^{2}\mathop{I^{p}}\limits^{1}y_{i'}+\mathop{I_{i'}}\limits^{2}y^{p})\frac{\partial f}{\partial x^{p}}+f^{-1}(\mathop{I^{p'}_{~;i'}}\limits^{2}F_{2}^{2}+
\mathop{I^{p'}}\limits^{2}y_{i'}+\mathop{I_{i'}}\limits^{2}y^{p'})\frac{\partial f}{\partial x^{p'}}\nonumber\\
+c(x)f^{2}F\mathop{I_{i'}}\limits^{2}=0.\label{thisoJ2}
\end{numcases}
Differentiating \eqref{thisoJ1} with respect to $y^{i'}$, we get
\begin{equation}\label{thisoJ11}
2f\mathop{I^{p}_{~;i}}\limits^{1}\frac{\partial f}{\partial x^{p}}y_{i'}+c(x)f^{2}F^{-1}y_{i'}\mathop{I_{i}}\limits^{1}=0,
\end{equation}
contracting \eqref{thisoJ11} with $y^{i'}$, we get
\begin{equation}\label{thisoJ111}
2f\mathop{I^{p}_{~;i}}\limits^{1}\frac{\partial f}{\partial x^{p}}F_{2}^{2}+c(x)f^{2}F^{-1}F_{2}^{2}\mathop{I_{i}}\limits^{1}=0.
\end{equation}
Transforming \eqref{thisoJ111}, then substituting it into \eqref{thisoJ1}, \eqref{thisoJ1} can be simplified as
\begin{equation}\label{thisoJ1111}
\mathop{J_{i}}\limits^{1}+c(x)(F-\frac{1}{2}f^{2}F^{-1}F_{2}^{2})\mathop{I_{i}}\limits^{1}=0.
\end{equation}
Differentiating \eqref{thisoJ1111} with respect to $y^{i'}$, we obtain
\begin{equation}\label{thisoJ11111}
c(x)\frac{1}{2}f^{4}F^{-3}F_{2}^{2}y_{i'}\mathop{I_{i}}\limits^{1}=0,
\end{equation}
contracting \eqref{thisoJ11111} with $y^{i'}$, we obtain
\begin{equation}\label{thisoJ111111}
c(x)\frac{1}{2}f^{4}F^{-3}F_{2}^{4}\mathop{I_{i}}\limits^{1}=0.
\end{equation}
Similarly, by \eqref{thisoJ2}, we can obtain
\begin{equation}\label{thisoJ1111111}
c(x)f^{4}F^{-3}F_{1}^{4}F_{2}^{2}\mathop{I_{i'}}\limits^{2}=0.
\end{equation}
Notice that $F$ is non-Riemannian, according to \eqref{thisoJ111111} and \eqref{thisoJ1111111}, hence $c(x)=0$, which means that $F$ is a weakly Landsberg metric.
\end{proof}

\begin{flushleft}
\Large\textbf{Acknowledgment}
\end{flushleft}
This work is supported by National Natural Science Foundation of China (Grant Nos. 12261088, 11761069).


\begin{thebibliography}{999999}
\bibitem{S} Z. Shen, Differental geometry of spray and Finsler spaces, Kluwer Academic Publishers, Dordrecht, The Netherlands, 2001.
\bibitem{Bao} D. Bao, On two curvature-driven problems in Riemann-Finsler geometry, Advanced Studies in Pure Mathematics. 48 (2007) 19-71.
\bibitem{ZWL} S. Zhou, J. Wang, B. Li, On a class of almost regular Landsberg metrics, Science China Mathematics. 62 (5) (2019) 935-960.
\bibitem{Shen} Z. Shen, On a class of Landsberg metrics in Finsler geometry, Canadian Journal of Mathematics. 61 (6) (2009) 1357-1374.
\bibitem{BO} R. L. Bishop, B. O'Neill, Manifolds of negative curvature, Transactions of the American Mathematical Society. 145 (1969) 1-49.
\bibitem{GSA1} G. S. Asanov, Finslerian extensions of Schwarzschild metric, Fortschritte der Physik/Progress of Physics. 40 (7) (1992) 667-693.
\bibitem{GSA2} G. S. Asanov, Finslerian metric functions over the product $R\times M$ and their potential applications, Reports on Mathematical Physics. 41 (1) (1998) 117-132.
\bibitem{HR1} A. B. Hushmandi, M. M. Rezaii, On warped product Finsler spaces of Landsberg type, Journal of mathematical physics. 52 (9) (2011) 093506.
\bibitem{PTN1} E. Peyghan, A. Tayebi, B. Najafi, Doubly warped product Finsler manifolds with some non-Riemannian curvature properties, Annales Polonici Mathematici. 3 (105) (2012) 293-311.
\bibitem{HC} Y. He, C. Zhong, On doubly warped product of complex Finsler manifolds, Acta Mathematica Scientia. 36 (6) (2016) 1747-1766.
\bibitem{CBY} B. Y. Chen, Geometry of submanifolds and its applications, Tokyo: Science University of Tokyo, III, 1981.
\bibitem{PR} R. Ponge, H. Reckziegel, Twisted products in pseudo-Riemannian geometry, Geometriae Dedicata. 48 (1) (1993) 15-25.
\bibitem{C} B. Y. Chen, Twisted product CR-submanifolds in Kaehler manifolds, Tamsui Oxford Journal of Mathematical Sciences. 16 (2) (2000) 105-121.
\bibitem{KPS} L. Kozma, I. R. Peter, H. Shimada, On the twisted product of Finsler manifolds, Reports on Mathematical Physics. 57 (3) (2006) 375-384.
\bibitem{PTN2} E. Peyghan, A. Tayebi, F. L. Nourmohammadi, On twisted products Finsler manifolds, ISRN Geometry. (2013) 732432.
\bibitem{Xiao1} W. Xiao, Y. He, X. Lu, X. Deng, On doubly twisted product of complex Finsler manifolds, Journal of Mathematical Study. 55 (2) (2022) 158-179.
\bibitem{Deng1} X. Deng, Y. He, N. Zhang, Berwald doubly twisted product Finsler metric, Journal of Xinjiang Normal University. 40 (2) (2021) (in Chinese).
\bibitem{Deng2} X. Deng, Y. He, Q. Ni, Weakly Berwald doubly twisted product Finsler metrics, Pure Mathematics. 11 (7) (2021) 1389-1399. (in Chinese).
\bibitem{Xiao2} W. Xiao, Y. He, C. Tian, J. Li, Complex Einstein-Finsler doubly twisted product metrics, Journal of Mathematical Analysis and Applications. 509 (2) (2022) 125981.
\bibitem{Xiao3} W. Xiao, Y. He, S. Li, Q. Ni, Locally conformally flat doubly twisted product complex Finsler manifolds, Journal of Mathematics. (2022)
\bibitem{Xiao4} W. Xiao, Y. He, X. Deng, J. Li, Locally dually flatness of doubly twisted product complex Finsler manifolds, Advance in Mathematics. (2022) (in Chinese)
\bibitem{BCS} D. Bao, S. S. Chern, Z. Shen, An introduction to Riemann-Finsler geometry. vol. 200 of Graduate Texts in Mathematics, Springer, New York, NY, USA, 2000.
\bibitem{F} P. Finsler, Uber Kurven and Flachen in allgemeinen Raumen, (Dissertation Gottingan, 1918), Birkhauser Verlag. 31 (10) (1951) 319.
\bibitem{Cartan} E. Cartan, Les espaces de Finsler, Actualites Scientifiques et Industrielles. Paris: Hermann. (1934) 79.
\bibitem{LS} B. Li, Z. Shen, On a class of weak Landsberg metrics, Science in China Series A: Mathematics. 50 (4) (2007) 573-589.
\bibitem{CSD} X. Chen, Z. Shen, On Douglas metrics, Publicationes Mathematicae Debrecen. 66 (3-4) (2005) 503-512.
\end{thebibliography}
\end{document}